\documentclass[11pt]{amsart}

\usepackage{amsmath,amsthm,amscd,amsfonts,amssymb}
\usepackage[all]{xy}
\usepackage[ansinew]{inputenc} 
\usepackage{dsfont}

\swapnumbers 

\newtheorem{thm}{Theorem}[section]
\newtheorem{teo}[thm]{Theorem}
\newtheorem{cor}[thm]{Corolary}

\newtheorem{prop}[thm]{Proposition}
\newtheorem{lemaN}[thm]{Lemma}
\theoremstyle{definition}
\newtheorem{defi}[thm]{Definition}
\newtheorem{obs}[thm]{Remark}
\newtheorem{obss}[thm]{Remarks}
\newtheorem{notacion}[thm]{Notation}
\newtheorem{defis}[thm]{Definition}
\newtheorem{coment}[thm]{Coment}
\newtheorem{General notions}[thm]{General notions}
\newtheorem{Los jets de orden 1}[thm]{Jets of order 1}
\newtheorem{Terminologia y notaciones}[thm]{Terminology and notation}
\newtheorem{Notaciones y terminologia}[thm]{Notation and terminology}
\newtheorem{nada}[thm]{}
\newtheorem{El ejemplo de los espacios cotangentes de orden superior}[thm]{The higher order cotangent spaces example}
\newtheorem{Aplicacion tangente a un morfismo en un jet}[thm]{Tangent map to a morphism at a jet}
\newtheorem{Dual de la aplicacion tangente}[thm]{Dual of a tangent map}
\newtheorem{El ejemplo de los jets de orden 1}[thm]{Order 1 jets example}
\newtheorem{Functorialidad}[thm]{Functoriality}
\newtheorem{Modulo tangente en un Apunto}[thm]{Tangent module at an $A$-point}
\newtheorem{El teorema es un caso particular del teorema de Weil}[thm]{Theorem \ref{tangenteApuntos} is a particular case of Weil theorem}
\newtheorem{Producto cartesiano de Apuntos Prolongacion de grupos de Lie y fibrados asociados}[thm]{Cartesian product of $A$-points. Prolongatons of Lie groups and associated fiber bundles}
\newtheorem{Un ejemplo el grupo tangente a un grupo de Lie}[thm]{An example: the tangent group of a Lie group}
\newtheorem{La estructura diferenciable de MA}[thm]{The differentiable structure of $\M^A$}
\newtheorem{Prolongacion de ideales}[thm]{Prolongation of ideals}
\newtheorem{Prolongacion de jets Relacion con los sistemas de contacto}[thm]{Prolongation of jets. Relationship with the contact system}

\numberwithin{equation}{section}

\def\C{\mathcal {C}}
\def\A{{\mathcal{C}}^\infty(M)}
\def\B{\mathcal {C}^\infty(N)}
\def\M{\sl{\mathds{M}}}
\def\N{\sl{\mathds{N}}}
\def\T{\sl{\mathds{T}}}
\def\X{\sl{\mathds{X}}}
\def\Z{\sl{\mathds{Z}}}
\def\DD{\sl{\mathds{D}}}
\def\D{\mathcal {D}}

\def\R{\mathbb R}
\def\m{\mathfrak m}
\def\p{{\mathfrak p}}
\def\q{\mathfrak q}
\def\To{\longrightarrow}
\def\vphi{\varphi}
\def\Der{\textrm{Der}\,}
\def\cM{\check M}
\def\cU{\check U}

\begin{document}

\title[Primary spectrum of $\A$ and jets theory]
 {Primary spectrum of $\A$ and jets theory}

\author{R. J.  Alonso-Blanco and J.  Mu\~{n}oz-D{\'\i}az}

\address{Departamento de Matem\'{a}ticas,
  Universidad de Salamanca, Plaza de la Merced 1-4, E-37008 Salamanca,
  Spain.}
\email{ricardo@usal.es, clint@usal.es}

\maketitle

\begin{abstract}
We consider, for each smooth manifold $M$, the set $\M$ comprised by all the primary ideals of $\A$ which are closed and whose radical is maximal. The classical Lie theory of jets (jets of submanifolds) must be extended to $\M$ in order to have nice functorial properties. We will begin with the purely algebraic notions, referred always to the ring $\A$. Subsequently it will be introduced the differentiable structures on each jets space of a given type. The theory of contact systems, which generalizes the classical one, has a part purely algebraic and another one which depends on the differentiable structures.
\end{abstract}

\bigskip

\vskip 1cm

\section*{Introduction}

Jets spaces appear in the work of S. Lie. The ``elements'' of his contact transformations are 1-jets of hypersurfaces. In the theory of differential invariants (see \cite{Lie}, volume 1, \S 130, for example) each group of transformations is prolonged from the base manifold to all the spaces of jets of submanifolds, in order to obtain all of the differential invariants of the group. Without stopping in formalizing the theory, it is clear that Lie thinks about jets as points and  about spaces of jets of a given type as ``prolongations'' of the given base-manifold. Even the notation he uses for the coordinates are the current ones.

Since the 1950s jets spaces are systematically used, considered its elements as jets of morphisms of manifolds and, more particularly, as jets of sections of fiber bundles (\cite{Saunders}, for example). Although frequently this view is convenient in many applications, facilitating calculations in coordinates, the point of view of Lie (jets of submanifolds) is preferable when one wants to put the theory into a general way, looking at the demands of the structure, rather than applications.

Given a smooth manifold $M$, a closed submanifold $X$ of dimension $m$ and a point $p\in X$, the jet of order $\ell$ of $X$ at $p$ is the ideal $I(X)+\m_p^{\ell+1}$ of $\A$ (we have denoted by $I(X)$ the ideal of functions vanishing on $X$, and by $\m_p$ the one consisting of functions null at $p$). The factor ring of $\A$ by the mentioned ideal is isomorphic with $\R_m^\ell=\R[x^1,\dots,x^m]/(x^1,\dots,x^m)^{\ell+1}$, the ring of Taylor expansions of $m$ variables truncated at order $\ell$, and the law associating with each $f\in\A$ its factor class in $\R_m^\ell$ is the Taylor expansion of order $\ell$ of $f|_X$ at $p$ (with respect to a choice of coordinates depending on the chosen isomorphism between the factor ring and  $\R_m^\ell$). If we denote by $\M_m^\ell$ the set of ideals of $\A$ whose remainder quotient rings are isomorphic to $\R_m^\ell$, $\M_m^\ell$ is identified with the jets of order $\ell$ of $m$-dimensional submanifolds of $M$, in the sense of Lie.

The assignment $M\rightsquigarrow\M_m^\ell$ is not functorial. This is the first reason to a good general theory of jets must to take into consideration not only jets of submanifolds (the ``classical'' jets of Lie), but, at least, all of the primary ideals of $\A$ which are closed for the topology of that ring and have maximal radical. That set, $\M$, is not a manifold, but the union of disjoint manifolds, one for each  ``Weil algebra'' (= local, rational and finite dimensional, as a $\R$-vector space, $\R$-algebra). The assignation $M\rightsquigarrow\M$ is already functorial

In \cite{Weil} Weil introduced a notion which generalizes that the point of a manifold: for each manifold $M$ and each Weil algebra $A$, an  $A$-point of $M$ is an morphism $p^A\colon\A\to A$. In particular, when $A=\R$, the $\R$-points are the usual points of $M$; for $A=\R_1^1$, the $A$-points are the tangent vectors; when $A=\R_n^1$ ($n=\dim M$) the ``regular'' $A$-points (exhaustive morphisms) are the frames on the $T_pM$. In general, in the currently usual terminology,  $\R_m^\ell$-points are jets of order $\ell$ of differentiable maps from $\R^m$ to $M$: $J^\ell_o(\R^m,M)=M_m^\ell$, in Weil notation.

For each Weil algebra $A$, the set of $A$-points of $M$ is a manifold $M^A$, and we have a canonical map $M^A\to\M$, which assigns with each $A$-point $p^A$ its kernel $\p=\textrm{Ker}\,p^A$. If we denote by $\cM^A$ the set of those $p^A$ which are exhaustive, its image $\M^A$ in $\M$ is another manifold, and the canonical map $\cM^A\to\M^A$, $p^A\mapsto\p=\textrm{Ker}\,p^A$, is a principal fiber bundle with structural group $\textrm{Aut}\,A$.

The relationship between Weil $A$-points and primary ideals of $\A$ is the same as the existent between points of algebraic manifolds in the classical  \cite{Foundations} and the current  Grothendieck schemes (which germinal idea is already in the work of Dedekind (\cite{Dirichlet} supplement XI,\cite{Dedekind-Weber}), contemporary of Lie).

In this work we present first the algebraic aspects of jets theory and, subsequently, the differentiable structures. In principle, all of the notions are referred to the ring $\A$. After of introducing the differentiable structure on each jets space $\M^A$, the new manifolds ones can play the role of the initial  $M$; for example, in the prolongation of systems of partial differential equations.

In the bibliography are cited a collection of papers where the consideration of jets of $M$ as ideals of $\A$ has been applied  to several problems, most of them related to works of Lie.

In order to make the exposition self-contained we have detailed some results included in previous publications (in particular, \cite{A1}, \cite{AM}, \cite{Weil})

\section{Jet spectrum of a manifold}

\begin{General notions}

 Let $M$ be a smooth manifold, $\A$ its ring of infinitely differentiable functions. We will call \emph{jet} of  $M$ to each ideal
    $\p$ of $\A$ such that the quotient $\A/\p$ is a finite dimensional, as an $\R$-vector space, rational and local $\R$-algebra; in short, a \emph{Weil algebra}.

We will denote by $\mathds{M}$ the set comprised by all jets of the manifold $M$ and call it \emph{jets spectrum} of $\A$ or, shortly, \emph{jet spectrum} of $M$.

Given a morphism of manifolds $\varphi\colon M\to N$, the morphism of rings $\varphi^*\colon \B\to\A$ determines a map, again denoted again $\varphi\colon\M\to\N$, which assign to the  jet $\p\in\M$ the ideal $\varphi(\p)=\varphi^{*-1}(\p)$ of $\B$, which is also a jet, because  $\B/\varphi(\p)$ is injected by means of $\varphi^*$ as a subalgebra of $\A/\p$.

The composition of morphisms of manifolds  corresponds to the composition of maps between jet spectra. Therefore, the assignation  $M\rightsquigarrow\M$ is a covariant functor from the category of manifolds to the category of sets.

Each point $p\in M$ determines a maximal ideal $\m_p$ of $\A$: the ideal of the $f\in\A$ such that $f(p)=0$; it holds $\A/\m_p=\R$. The converse is a classical result: every maximal ideal of $\A$ whose field of remainders is $\R$ comes from a point $p\in M$. In this way, manifold $M$ is recoverable as the $\R$-spectrum of $\A$. The topology of $M$ is that of Zariski as a part of the maximal spectrum of $\A$.

Let $\p$ be an arbitrary  jet of $M$; the quotient $\A/\p$ is a Weil algebra, whose maximal ideal has the form $\m/\p$, where $\m$ is a maximal ideal of $\A$ whose residue field is $\R$; for this reason, $\m$ is of the form $\m_p$ form some $p\in M$. The jet $\p$ is contained in (a unique) maximal ideal of the form $\m_p$, with $p\in M$. Following the terminology suggested by that of  Weil \cite{Weil}, we will say that $\p$ is a \emph{jet near}  $p$ or a \emph{jet} at $p$. Thus, we have a canonical map ${\M}\to {M}$, ($\p\mapsto p\,|\,\m_p\supseteq\p$), which induces the identity  on $M\subseteq\M$.

If $A$ is a Weil algebra, we call \emph{$A$-point of $M$} each morphism of  $\R$-algebras $p^A\colon \A\to A$ (Weil \cite{Weil}). If $\m_A$ is the maximal ideal of $A$, it holds $A/\m_A=\R$. The composition of morphism $p^A$ with the morphism quotient by $\m_A$ gives a morphism $\A\to\R$, and so a point $p\in M$. We will say that the  \emph{$A$-point $p^A$ is near $p$}.

Given a Weil algebra $A$, the powers of its maximal ideal $\m_A$ define a decreasing sequence of finite dimensional $\R$-vector spaces; such a sequence must be stationary: for some index $\ell$ is $\m_A^{\ell+1}=\m_A^{\ell+2}$. By the ``lema of Nakayama'' is, then,  $\m_A^{\ell+1}=(0)$. The first $\ell$ for which $\m_A^{\ell+1}=(0)$ is called \emph{order of $A$}. The dimension of the $\R$-vector space $\m_A/\m_A^2$ is named \emph{width of $A$}. The width of  $A$ is the minimum  number of generators of $A$ as an $\R$-algebra.

For each jet $\p$ of $M$, the  \emph{order} of $\p$ and the \emph{width} of $\p$ will be those of the Weil algebra $\A/\p$.

Given an ideal $I$ of $\A$ and a point $p\in M$, such that $I\subseteq \m_p$, we will call \emph{jet of order $\ell$ of $I$ at $p$} the ideal $I+\m_p^{\ell+1}$; such an ideal consists of the whole of $f\in\A$ whose Taylor expansion order $\ell$ at $p$ coincides with the one of some function in the ideal $I$. The classical theorem of Whitney \cite{Whitney} can be stated by saying that each ideal $I$ which is closed in the topology of $\A$, is the intersection of all the jets containing it. The converse is immediate: powers $\m_p^k$ are closed ideals of finite codimension in $\A$, from which all jets and their intersections are also closed.

Denote by $\R_m^\ell$ the Weil algebra $\R[x^1,\dots,x^m]/(x^1,\dots,x^m)^{\ell+1}$. If $X$ is a non singular, closed submanifold of dimension $m$ in $M$, and $p\in X$, we can take local coordinates $(x^1,\dots,x^m,y^1,\dots,y^r)$ in a neighborhood of $p$ in $M$ in such a way that local equations of $X$ will be $y^1=0,\dots,y^r=0$. We see, then, that $\A/I(X)+\m_p^{\ell+1}\simeq \R_m^\ell$. Jets $\p\in\M$ whose quotient rings $\A/\p$ are isomorphic to algebras   $\R_m^\ell$ will be called \emph{classical jets}. They are the jets of arbitrary order of closed (or locally closed) submanifolds of $M$. In particular, if $\pi\colon M\to N$ is a regular projection, for each section $\sigma\colon N\to M$, the image $\sigma(N)$ is a closed submanifold of  $M$ whose jet of order $\ell$ at each $p=\sigma(q)$ is usually named \emph{jet of order $\ell$ of the section $\sigma$ at the point $q$}, and denoted by  $j_q^\ell(\sigma)$. They are the classical jets of sections of fiber bundles.
\end{General notions}

\begin{Los jets de orden 1}\label{orden1}
 If $A$ is a Weil algebra of order $1$, we have $\m_A^2=0$ and $A\simeq\R_m^1$, where $m$ is the width of $A$. Hence, it is derived that each jet $\p$ of order 1 of the manifold $M$ is of the form $I(X)+\m_p^2$, where $X$ is a submanifold of $M$ which passes through the point $p$ and whose dimension equals the width of $\p$. The quotient $\p/\m_p^2\subseteq\m_p/\m_p^2\simeq T^*_pM$ is called \emph{contact system of $\M$ at $\p$} and completely characterizes the jet $\p$. The subspace $L_\p\subseteq T_pM$ annihilated by the contact system is the space tangent to $X$ at $p$. In this way, jets order $1$ of $M$ are conceivable as the subspaces of the tangent spaces of $M$ at each one of its points. Given the subspace $L_\p\subseteq T_pM$, the ideal $\p$ of $\A$ is the set of functions $f$ such that  $f(p)=0$ and $D_pf=0$ for all $D_p\in L_\p$.

The definition of the contact systems at every point $\p\in\M$ is one of the objectives of this work.
\end{Los jets de orden 1}

\begin{Terminologia y notaciones}\label{terminologia}

When an ideal $I$ of $\A$ will be contained in a jet $\p$, we will say that \emph{$I$ passes trough $\p$}. When $I=I(X)$ is the ideal of a submanifold (singular or not) and $I(X)\subseteq\p$, we will say that \emph{the manifold  $X$ passes through $\p$}.

If $A$ is a Weil algebra and it holds $\A/\p\simeq A$, we will say that  $\p$ is a \emph{jet of type $A$}. The set of all jets of $M$ of type $A$ will be denoted by $\M^A$ and also $J^A(M)$, when this notation is more handy. The set of jets of type $\R_m^\ell$ will be denoted by $\M_m^\ell$ or $J_m^\ell(M)$. When it be convenient to indicate that that a jet is of type  $A$, we will denote it by a superindex $\p^A$ or $\p_m^\ell$, in the cases of classical jets.

The set of all  $A$-points of $M$ will be denoted by $M^A$. When $A=\R_m^\ell$, we will put $M_m^\ell$.

For each Weil algebra $A$ we have the canonical map $M^A\to \M$, $p^A\mapsto \p=\ker p^A$. In general, the morphism $p^A$ is not exhaustive, for which its kernel  $\p$ can be not of type $A$, although always will be of type ``a subalgebra of A''.

When $p^A\colon\A\to A$ is exhaustive, we will say that $p^A$ is a \emph{regular $A$-point}; in such a case, $\p^A=\ker p^A$ is a jet of type $A$. We will denote by $\check{M}^A$
the set of regular  $A$-points $M$. The canonical map $M^A\overset{\ker}\to\M$ sends $\check{M}^A$ onto ${\M}^{A}$.

For a given Weil algebra $A$, each morphism of manifolds $\vphi\colon M\to N$ induces canonically a map of sets (same notation) $\vphi\colon M^A\to N^A$, defined by  $\vphi(p^A)=p^A\circ \vphi^*$. The assignation $M\rightsquigarrow M^A$ is a covariant functor from the category of manifolds to the category of sets (at this moment, we do not dispose yet of a differentiable structure for $M^A$) and the diagrams
\begin{equation*}
\xymatrix{
M^A\ar[r]^-{\vphi} \ar[d]_-\ker & N^A\ar[d]^-\ker\\
\M\ar[r]_-\vphi & \N
}
\end{equation*}
are commutative.

Obviously, each morphism of Weil algebras $\alpha\colon A\to B$ determines maps $\alpha\colon M^A\to M^B$. Fixing the manifold $M$, the assignation $A\rightsquigarrow M^A$ is a covariant functor from the category of Weil algebras to the category of sets (in fact, to that of manifolds, when we will dispose of the above mentioned  differentiable structures).

Let us suppose that $\alpha$ is exhaustive; put $I=\ker\alpha$. Let $p^A\in \check{M}^A$ and $\p^A=\ker p^A\in\M^A$. The image of $p^A$ by $\alpha$ is the $B$-point $p^B=\alpha\circ p^A$, whose kernel is $\p^B=(p^A)^{-1}I$ which, in general, depends on $p^A$, and not only on $\p^A$. For this reason, in general, the map $\check{M}^A\to \check{M}^B$, canonically associated to the morphism $\alpha$, does not project on a map $\M^A\to \M^B$. Let us see which is the condition for the existence of such a projection: every regular $A$-point $\overline{p}^A$, with the same jet as $p^A$, is of the form $\overline{p}^A=g\circ p^A$, for some $g\in\textrm{Aut}\,A$. Its image under $\alpha$ is $\overline{p}^B=\alpha\circ g\circ p^A$, whose jet is $\overline{\p}^B=(p^A)^{-1}(g^{-1}I)$, in such a way that the condition to be $\overline{\p}^B=\p^B$ is (since $p^A$ is exhaustive) to have $g(I)=I$. Therefore, the condition for the map $\check{M}^A\to\check{M}^B$ to go down to level of jets, completing and making commutative the diagram
\begin{equation*}
\xymatrix{
\check{M}^A\ar[r]^-{\alpha} \ar[d]_-\ker & \check{M}^B\ar[d]^-\ker\\
\M^A\ar@{-->}[r] & \M^B,
}
\end{equation*}
  is that the ideal $I=\ker\alpha$ to be stable under the action of the group $\textrm{Aut}\,A$. In such a case, we have a canonical group morphism $\textrm{Aut}\,A\to\textrm{Aut}\,B$, whose kernel is the group $G_\alpha=\{g\in\textrm{Aut}\,A\,|\,\alpha\circ g=\alpha\}$.

\end{Terminologia y notaciones}

\section{Tangent and cotangent modules at a jet}

The justification of the following definitions will appear when we will study the differentiable structures on jet spaces.

\begin{defi}\label{modulotangente}
For a given $\p\in\M$ we will call \emph{tangent module to the manifold $M$ at $\p$}, and we denote it by $T_\p M$, the quotient
\begin{equation*}
T_\p M:=\Der(\A,\A/\p)/\Der(\A/\p,\A/\p)
\end{equation*}
Let us denote $\T M:=\{T_\p M\,|\,\p\in\M\}$.
\end{defi}

Let $\alpha\colon A\to B$ be an exhaustive  morphism of Weil algebras, whose kernel $I$ is stable under the group $\textrm{Aut}\,A$. The Lie algebra of this group is $\Der(A,A)$. We deduce that, for every derivation  $\delta\colon A\to A$, it holds $\delta I\subseteq I$, so that  $\delta$ projects onto a derivation $\overline\delta\colon B\to B$.

On the other hand, by considering $A=\A/\p^A$, via a $p^A\in\check M^A$ whose kernel is  $\p^A$,  and analogously for $B$, via $p^B=\alpha\circ p^A$, we see that the natural map
\begin{align*}
\Der_{p^A}(\A,A)&\To\Der_{p^B}(\A,B),\\
   D_{p^A} &\longmapsto\alpha\circ D_{p^A}
\end{align*}
sends $\Der(A,A)$ to $\Der(B,B)$, so that it passes to the quotient, giving $T_{\p^A}M\to T_{\p^B}M$.

\begin{Notaciones y terminologia}\label{terminologia2}

Let us put $\D(M)=\A$-module of all the tangent fields on $M$. When there is no risk of confusion, we will put $\D$ instead of $\D(M)$.

For each $D\in\D$ and each $\p\in\M$, we will denote by  $D_\p$ the derivation from $\A$ to $\A/\p$ which is the composition of the derivation  $D$ with the pass to the quotient. We will call $D_\p$ the \emph{value of the field $D$ at the jet  $\p$}.
 The class of $D_\p$ in the quotient module $T_\p M$ will be denoted by $\DD_\p$; so, the derivation $D_\p$ is a \emph{representative of the tangent vector $\DD_\p$}.

For each ideal $I$ of $\A$ (in particular, for each jet) we will put $\D(I)=\{D\in\D\,|\,D(I)\subseteq I\}$. Fields $D\in\D(I)$ will be said to be  \emph{tangent to
$I$}. When $I$ is the ideal of a manifold $X$, the tangent fields to $I(X)$ will be said to be \emph{tangent to $X$}.

We will say that a field $D\in\D$ is \emph{tangent to the ideal $I$  at the jet $\p$} when  $I$ passes through $\p$ and, in addition, $DI\subseteq\p$. When $I$ is the ideal of a manifold $X$ and $D$ is tangent to $I$ at $\p$, we will say that \emph{$D$ is tangent to $X$ at $\p$}.

According to the Whitney theorem \cite{Whitney}, the field $D$ is tangent to the closed ideal $I$ if and only if it is tangent to $I$ at every jet by  which it passes through.
\end{Notaciones y terminologia}

\begin{prop}\label{tangentejets}
All derivation $\A\to\A/\p$ is the value at $\p$ of a field $D\in\D$. Derivations from $\A/\p$ to itself are of the form $D_\p$, for an appropriate $D\in\D(\p)$. The module tangent to $M$ at $\p$ is
$$T_\p M=\D(M)/\D(\p).$$
\end{prop}
\begin{proof}
Let $\ell$ be the order of $\p$, $n$ the dimension of $M$, $p\in M$ the point of which  $\p$ is nearby. Then, $\A/\p$ is a quotient of $\A/\m_p^{\ell+1}\simeq\R_n^\ell$, and all derivation $\A\to\A/\p$ sends $\m_p^{\ell+2}$ to $0$, so it is a a derivation from $\A/\m_p^{\ell+2}\simeq\R_n^{\ell+1}$ to a quotient of $\R_n^{\ell+1}$. It is easy to see that this derivation cames from a derivation of $\R[x^1,\dots,x^n]$ to itself, and from that its derived the first point of the statement. The remainder is a consequence of it.
\end{proof}

\begin{nada}

Each  $f\in\p$ determines a morphism which we will denote by $d_\p f$:
$$d_\p f\colon T_\p M\to\A/\p,\quad \DD_\p\mapsto D_\p f,$$
where $D_\p$ is an arbitrary representative of $\DD_\p$. If $D_\p f=0$ for all $f\in\p$, it holds $\DD_\p=0$, according to definitions, hence the collection
$$\{d_\p f\,|\,f\in\p\}\subseteq\textrm{Hom}_{\A}(T_\p M,\A/\p)$$
distinguishes points in $T_\p M$.

The condition for $d_\p f=0$ is that for all tangent field $D$ it holds $Df\in\p$.
\end{nada}

\begin{notacion}
For each jet $\p$, we will put:
$$\hat\p=\{f\in\p\,|\,\forall D\in\D,\, Df\in\p\}=\textrm{Ker}\,d_\p.$$
\end{notacion}

It is clear that $\hat\p$ is an ideal of $\A$ and that
$$\p^2\subseteq \hat\p\subseteq \p,$$
for which $\hat\p$ is also a jet, at the same point $p$ as $\p$.

\begin{defi}
We will call \emph{module cotangent to $M$ at $\p$} to
$$T_\p^*M:=\p/\hat\p\subseteq\textrm{Hom}(T_\p M,\A/\p).$$
The inclusion assigns to the class $[f]_{\textrm{mod}\,\hat\p}$ the $d_\p f$.
In general, that inclusion is not exhaustive.
\end{defi}

\begin{El ejemplo de los espacios cotangentes de orden superior}
In the simplest case, when $\p=\m_p$, it is clear that $T_\p M=T_pM$ and $\hat\m_p=\m_p^2,$ so that $T_\p^*M=\m_p/\m_p^2=T_p^*M$.

 Now let $\p=\m_p^{\ell+1}$. It is easy to see that $\D(\p)=\m_p\D$ in this case. Then $T_\p M=\D/\m_p\D=T_pM$ again. But now, $T_\p^*M$ does not coincide with $T^*_p M$, as we will see later.

It is easily checked that $\hat\p=\m_p^{\ell+2}$, so that $T_\p^*M=\m_p^{\ell+1}/\m_p^{\ell+2}$ is $\simeq\{$homogeneous forms of degree $\ell+1$ in the local coordinates around $p$\}. Each one of these forms $F(x^1,\dots,x^n)$ determines the morphism from $T_\p M$ to $\A/\p$ which sends each $\DD_\p$ (class of the field $D$) to $[DF]_{\textrm{mod}\,\m_p^{\ell+1}}$, value that depends only on the vector $D_p\in T_p M$. The correspondence $M\to\M_n^\ell$, $p\mapsto\m_p^{\ell+1}$, is bijective and when we will study the differentiable structures on jet spaces, we will proof that it is an isomorphism of manifolds. The tangent spaces are identified; by seeing only the structures of manifold, also the respective cotangent spaces at the corresponding point,  are identified. However, the degree $\ell$ distinguishes between those that we have called cotangent modules.

\end{El ejemplo de los espacios cotangentes de orden superior}

\begin{Aplicacion tangente a un morfismo en un jet}

Let $\varphi\colon M\to~N$
be a morphism of manifolds and $\varphi^*\colon\B\to \A$ the corresponding morphism between their rings of functions. For each jet $\p\in\M$, if $\q:=\varphi(\p)\in\N$ denotes its image by $\varphi$, we have an injection of rings $\varphi^*\colon \B/\q\to\A/\p$, by means of which we consider the first ring as a subring of the second one.

For each tangent field $D$ on $M$ and each $\p\in\M$ define the derivation of $\B$-modules
$$\varphi_*D_\p:=D_\p\circ\varphi^*\colon\B\to \A/\p.$$
When this derivation takes values in the subring $\B/\varphi(\p)$, we will say that \emph{there exists the map tangent to $\varphi$ at $\p$}, and we consider $\varphi_* D_\p$ as a derivation from $\B$ to $\B/\varphi(\p)$.
\end{Aplicacion tangente a un morfismo en un jet}

\begin{prop}\label{condiciontangente}
The necessary and sufficient condition for the existence of map tangent to the morphism $\varphi$ at the jet $\p$ is that, for all tangent field $D$ on $M$, it holds
$$D\varphi^*\B\subseteq\varphi^*\B+\p.$$
\end{prop}
\begin{proof}
Immediate.
\end{proof}

Let us suppose that there exists map tangent to $\varphi$ at $\p$. Let $D$ be a field on $M$ such that $\DD_\p=0$; then $D$ maps $\p$ to $\p$, and so, applies $\varphi^*(\varphi(\p))\subseteq\p$ into $\p$, and $\varphi_*D_\p$ sends $\varphi(\p)$ to $(0)$. Therefore, the class of $\varphi_*D_\p$ in $T_{\varphi(\p)}N$ is $0$. From here it is derived the:
\begin{prop}\label{tangenteinducida}
If there exists a map tangent to morphism $\varphi$ at the jet $\p$, such a map canonically induces a morphism of $\B$-modules
$$\varphi_*\colon T_\p M\To T_{\varphi(\p)}N.$$
\end{prop}

Condition \ref{condiciontangente} for the existence of tangent map obviously holds when $\varphi^*\B+\p$ is the whole of ring $\A$, this is to say, when the morphism of taking quotient  $\A\to\A/\p$ restricted to the subring $\varphi^*\B$ is exhaustive. That holds, for example, when $M$ is a submanifold of $N$, and also when $\varphi$ is a regular projection and $\p$ is the jet of a section of it.

For brevity, we will say that a jet is \emph{proper or regular for a subring} $\mathcal{A}\subseteq\A$ when $\mathcal{A}+\p=\A$. Hence, given a morphism $\varphi\colon M\to N$, there exists a map tangent to $\varphi$ at any jet $\p\in\M$ which are regular for $\varphi^*\B$.

\begin{Dual de la aplicacion tangente}
Assume the existence of a map tangent to the morphism $\vphi$ at the jet $\p$. Put $\q=\vphi(\p)$.

If $f\in\hat\q$, for every tangent field $D$ on $M$ it must be $(\vphi_* D_\p)f=0$, so that $D_\p\vphi^*f=0$, from which we get $\vphi^*f\in\hat\p$.
It follows that  $\vphi^*\hat\q\subseteq\hat\p$, that is to say, $\hat\q\subseteq\vphi^{*-1}\hat\p=\vphi(\hat\p)$. It is derived that the morphism of rings, $\vphi^*$, induces a morphism of $\B$-modules $\vphi^*\colon\q/\hat\q\To\p/\hat\p$.

For each $f\in\q$ and each tangent vector $\DD_\p\in T_\p M$, we have $\langle\vphi_*\DD_\p,\, d_{\vphi(\p)}f\rangle=\langle\DD_\p,\, d_{\p}\vphi^*f\rangle$
($=D_\p\vphi^*f$, for any representative $D_\p$ de $\DD_\p$), by taking into account the identification of $\B/\vphi(\p)$ with its image in $\A/\p$.
\end{Dual de la aplicacion tangente}

\section{The contact system on $\M$}\label{seccioncontacto}

For each pair of jets $\p\subseteq\p'$ of $\M$ and each $f\in\p$ let us define the map
$$d_\p'f\colon T_\p M\To\A/\p'$$
by putting $$\langle d'_\p f,\DD_\p\rangle=[\DD_\p f]_{\textrm{mod}\,\p'}.$$

We will call  \emph{contact system of the pair $(\p,\p')$} the module
$$\Omega_{(\p,\p')}=\{d'_\p f\,|\,f\in\p\}\subseteq\textrm{Hom}(T_\p M,\A/\p').$$

\begin{obs} Each $f\in\p$ defines  $d_{\p'}f\in\textrm{Hom}(T_{\p'} M,\A/\p')$ and, for each given field $D$, it holds
$$\langle d_{\p'} f,\DD_{\p'}\rangle=\langle d'_\p f,\DD_\p\rangle,$$
however, the datum $\DD_\p$ does not determines, in general, to $\DD_{\p'}$, because it could not exist the map between the tangent spaces
$T_\p M\To T_{\p'}M$ associated to taking the quotient $\A/\p\to\A/\p'$; such a map exists when it holds $\D(\p)\subseteq\D(\p')$, which, in general, is not the case. Therefore, it must to distinguish between $d'_\p$ and $d_{\p'}$.
\end{obs}

The contact system on $\M$ is defined by assigning to each jet $\p$ another one $\p'$ which contain it, as we will see now.

\begin{defis}\label{defcontacto}
 Let $A=\A/\p$ be a Weil algebra of order $\ell$ and width $m$. We define the \emph{Cartan system at $\p$} to be the submodule $\C_\p$ of $T_\p M$ generated by the tangent fields on $M$ which are tangent to some $m$-dimensional submanifold $X$ of $M$ passing through $\p$. Next, let us define the ideal $\p'$ \emph{derived of $\p$} to be:
$$\p':=\p+(D_\p\p\,|\,\DD_\p\in\C_\p)$$
and the \emph{contact system at $\p$} to be:
$$\Omega_\p:=\Omega_{(\p,\p')}=\{d'_\p f\,|\,f\in\p\}\subseteq\textrm{Hom}(T_\p M,\A/\p').$$
\end{defis}

\begin{El ejemplo de los jets de orden 1} (Compare with \ref{orden1})
A jet of order 1 and width $m$ is the jet of order 1 of a submanifold $X$ of dimension $m$ at one of its points $p$. Let us take local coordinates around $p$ in $M$, $(x^1,\dots,x^m,y^1,\dots,y^r)$, adapted to $X$ in such a way that $I(X)=(y^1,\dots,y^r)$, so that $\p=(y^1,\dots,y^r)+\m_p^2$. Local equations of an  $m$-dimensional submanifold  of $M$ passing trough $\p$ are of the form $y^j=f^j(x^1,\dots,x^m)$ with $f^j\in\m_p^2$. The local expression of a vector field tangent to that submanifold is a linear combination of
$$\frac{\partial}{\partial x^\alpha}+\frac{\partial f^j}{\partial x^\alpha}\frac{\partial}{\partial y^j},\quad (\alpha=1,\dots,m).$$
It is derived that $\p'=\m_p$. For each $f\in \p$,  $d'_\p f$ maps each $\DD_\p$ to $[Df]_{\textrm{mod}\,\p'}=D_p f$, where $D$ is an arbitrary field representative of the tangent vector $\DD_\p\in T_\p M$. It follows that $d'_\p f=d_p f$, the differential of $f$ at $p$. Hence, $\Omega_\p$ is the set of all $d_p f$, with $f\in\p$, this is to say, $\p/\m_p^2$. The subspace of $T_\p M$ which annihilates  $\Omega_\p$ is, therefore, the set of all $\DD_\p\in T_\p M$ whose projection $D_p$ on $T_p M$ annihilates the ideal $\p$; that projection is the subspace $L_\p\subseteq T_\p M$ of \ref{orden1}.
\end{El ejemplo de los jets de orden 1}

The preimage of $L_\p$ in $T_\p M$ contains the Cartan subspace $\C_\p$ because the latter is annihilated by $\Omega_\p$. Right away, we will see that $\C_\p$ is exactly the preimage of $L_\p$; that is a corollary of the following proposition:
\begin{prop}\label{contactoaniquilado}
$\C_\p$ is the submodule of $T_\p M$ annihilated by $\Omega_\p.$
\end{prop}
\begin{proof}
We can take local coordinates for $M$ around $p$, $\{x^1,\dots,x^m,y^1,\dots,y^r\}$ in such a way that:
$$\p=(y^1,\dots,y^r)+\m_p^{\ell+1}+\left(Q^h(x)\right),$$
 where the $Q^h$ are suitable polynomials in the $x$'s of degrees $\ge 2$ and $\le\ell$ (in case $\ell=1$, the $Q^h$ are 0). An $m$-dimensional  submanifold $X$ of $M$ passing through $\p$ also passes through $(y^1,\dots,y^r)+\m_p^2$, from which we see that it is locally parameterized by coordinates $x$. Its local equations must be of the form
$$y^i=f^i(x^1,\dots,x^m)+F^i(x^1,\dots,x^m),\quad i=1,\dots,r,$$
where $f^i\in\left(Q^h(x)\right)$ and $F^i\in\m_p^{\ell+1}$. Fields tangent to $X$ are combinations of the
$$\frac{\partial}{\partial x^\alpha}+\left(\frac{\partial f^i}{\partial x^\alpha}+\frac{\partial F^i}{\partial x^\alpha}\right)\frac{\partial}{\partial y^i}.$$

Deriving the elements of $\p$ by these fields it is seeing that
$$\p'=\p+\m_p^\ell+\left(\frac{\partial Q^h}{\partial x^\alpha}\right)=
  (y^1,\dots,y^r)+\m_p^\ell+(Q^h)+\left(\frac{\partial Q^h}{\partial x^\alpha}\right).$$

  Let $\DD_\p\in T_\p M$; let $D_\p$ be a derivation representative of $\DD_\p$.
  If $\DD_\p$ is killed by $\Omega_\p$, $D_\p$ sends $\p$ to $\p'$; in the local expression
    $$D= a^\alpha(x,y)\frac{\partial}{\partial x^\alpha}+b^j(x,y)\frac{\partial }{\partial y^j},$$
 that condition is equivalent to be $b^j\in\p'$ ($j=1,\dots,r)$.

By subtracting from $b^j$ terms in $\p$ (which contribute by fields tangent to $\p$), we can assume that each $b^j$ is a polynomial in the $x$ belonging to the ideal $\m_p^\ell+\left(\partial Q^h/\partial x^\alpha\right)$, and also that the term in $\m_p^\ell$ is homogeneous of degree $\ell$ in the $x$.

 Let us study separately each term of $D$:

 The $m$ fields $\partial /\partial x^\alpha$ are tangent to the manifold having as ideal $(y^1,\dots,y^r)$, which passes through $\p$.

 Let $f(x^1,\dots,x^m)$ homogeneous of degree $\ell$; let $F(x^1,\dots,x^m)$ homogeneous of degree $\ell+1$ with $\partial F/\partial x^1=f$. The field
 $\partial/\partial x^1+f(x^1,\dots,x^m)\partial/\partial y^j$ is tangent to the manifold of equations $y^1=0, \dots$, $y^j=F(x^1, \dots,x^m),\dots,$  $y^r=0$; so that its value at $\p$ belongs to $\C_\p$; thus, also $f\partial/\partial y^j$ do it.

 Finally, let $b^j\in\left(\partial Q^h/\partial x^\alpha\right)$; by deriving products and eliminating terms in $\p$, we can assume that $b^j=\partial H^{j\alpha}/\partial x^\alpha$ with $H^{j\alpha}\in\p$. The field
 $$D_1=\frac{\partial}{\partial x^1}+\frac{\partial H^{j1}}{\partial x^1}\frac{\partial}{\partial y^j}$$
  is tangent to the manifold of equations $y^j=H^{j1}(x^1,\dots,x^m)$ ($j=1,\dots,r$). It follows that $\left(\partial H^{j1}/\partial x^1\right)\partial/\partial y^j$ belongs to $\C_\p$. Analogously, when $x^1$ is replaced by other coordinates $x$. Finally we conclude that $D$ is the sum of fields whose values at $\p$ belong to $\C_\p$, so that $\DD_\p\in\C_\p$.
 \end{proof}

\begin{prop}\label{derivadappprima}
$\D(\p)\subseteq\D(\p')$. In consequence, there exists tangent map $T_\p M\to T_{\p'}M$ associated to the inclusion $\p\subseteq\p'$.
\end{prop}
\begin{proof}
Let us take local coordinates for $M$ around $p$ like in the previous proof.

A field $D=a^\alpha\partial/\partial x^\alpha+b^j\partial/\partial y^j$ is tangent to $\p$ if and only if $b^j\in\p$, $a^\alpha\in\m_p$ and $a^\alpha\partial Q^h/\partial x^\alpha\in\p$. Subtracting terms in $\p$, we can assume that the $a^\alpha$ are polynomials in the $x$'s.

In order to see that $D$ maps $\p'$ into $\p'$ we have to check that $D\left(\partial Q^h/\partial x^\alpha\right)\in\p'$:
$$D\,\frac{\partial Q^h}{\partial x^\beta}=
  a^\alpha\frac{\partial }{\partial x^\alpha}\frac{\partial Q^h}{\partial x^\beta}=
   \frac{\partial}{\partial x^\beta}\left(a^\alpha\frac{\partial Q^h}{\partial x^\alpha}\right)-
     \frac{\partial a^\alpha}{\partial x^\beta}\frac{\partial Q^h}{\partial x^\alpha}.$$

Since $a^\alpha{\partial Q^h}/{\partial x^\alpha}\in\p\cap\R[x^1,\dots,x^m]\subseteq\m_p^{\ell+1}+(Q^k),$
it follows that $D\left(\partial Q^h/\partial x^\alpha\right)\in\p'$.
\end{proof}

\begin{prop}
 If the width of $\p'$ equals that of $\p$, the map $T_\p M\to T_{\p'}M$ sends $\C_\p$  into  $\C_{\p'}$.
\end{prop}
\begin{proof}
 Every $m$-dimensional  manifold passing through $\p$, also passe through $\p'$.
 \end{proof}

 \begin{prop}\label{proyeccioncontacto1}
 Let $\p\in\M$ with width $m$. For each $m$-dimensional manifold $X$ which passes through $\p$ it holds
 $$T_{\p'}X=\pi_*\C_\p,$$
 where $\pi_*\colon T_\p M\To T_{\p'}M$ is the canonical projection.
 \end{prop}
 \begin{proof}
 The tangent map $T_\p X \to T_\p M$ (\ref{tangenteinducida}) identifies $T_\p X$ with the submodule of $T_\p M$ comprised by the values at $\p$ of the tangent fields on $M$ which are tangent to $X$. By definition of the Cartan system we get, thus, $T_\p X\subseteq\C_\p$, so that $\pi_*T_\p X\subseteq\pi_*\C_\p$; from Proposition \ref{derivadappprima} and the inclusion $\p\subseteq\p'$, it is derived that $\pi_*T_\p X=T_{\p'}X$. Therefore,  $T_{\p'}X\subseteq\pi_*\C_\p$. On the other hand, if $\DD_\p\in\C_\p$, for any field $D$ representative of $\DD_\p$ it holds $D\p\subseteq\p'$, so that $DI(X)\subseteq\p'$, which means that $\DD_{\p'}$ is tangent to $X$ at $\p'$.
 \end{proof}

 \begin{teo}\label{proyeccioncontacto}
  Let $\p\in\M$ of width $m$. Every $m$-dimensional manifold $X$ passing through $\p$ have at $\p'$ the same tangent space, namely, $\pi_*\C_\p$. In addition, $\C_\p=\pi_*^{-1}(T_{\p'}X)$.
  \end{teo}
 \begin{proof}
 The first point is \ref{proyeccioncontacto1}. It remains to show that the kernel of $\pi_*$ is contained in $\C_\p$: such a kernel is the set of values at $\p$ of the fields which map $\p'$ into $\p'$; these fields send $\p$ a $\p'$, so that their classes in $T_\p M$ are annihilated by $\Omega_\p$, thus \ref{contactoaniquilado} implies they belong to $\C_\p$.
 \end{proof}

 \begin{teo}\label{teoTaylor}
 Let $\p\in\M$ of width $m$. If it holds $\widehat{\p'}\subseteq\p$, a manifold $X$ of dimension $m$ that passes through $\p'$, but not through $\p$, has at $\p'$ a tangent space different of $\pi_*\C_\p$. In consequence, if $\p,\q\in\M$ are different, with width $m$ and hold $\widehat{\p'}\subseteq\p$, $\widehat{\q'}\subseteq\q$, the spaces $\pi_*\C_\p$ and $\pi_*\C_\q$ are different.
 \end{teo}
 \begin{proof}
 Again, let us take local coordinates around $p$ in $M$ in such a way that $$\p=(y^1,\dots,y^r)+\m_p^{\ell+1}+\left(Q^h(x)\right),$$ and let $X$ be an $m$-dimensional manifold passing through $\p'$, but not through $\p$. The local equations of $X$ take the form $y^j=f^j(x^1,\dots,x^m)$ ($j=1,\dots,r$), where at least one of the $f^j$'s is out of $\p$. The manifold $Y$ given by equations $y^1=0,\dots,y^r=0$ passes through $\p$, and the field $\partial/\partial x^\alpha$ is tangent to $Y$, from which it follows that $[\partial/\partial x^\alpha]_{\p'}$ is into $T_{\p'}Y=\pi_*\C_\p$ (by \ref{proyeccioncontacto1}). If it holds
 $[\partial/\partial x^\alpha]_{\p'}\in T_{\p'}X$ it follows  $\left(\partial/\partial x^\alpha\right)I(X)\subseteq\p'$, and then $\left(\partial f^j/\partial x^\alpha\right)\in\p'$. If this is true for all $\alpha=1,\dots,m$, then it would be $f^j\in\widehat{\p'}\subseteq\p$, against the hypothesis. In this way, $T_{\p'}X\ne\pi_*\C_\p$.

 The second point in the theorem is a consequence of the first one and the fact of each jet $\p$ of width $m$, as an ideal of $\A$, is the sum of all the ideals of $m$-dimensional submanifolds passing through $\p$.
   \end{proof}

 \begin{defi}
 The map $\M\To\textrm{Grass}(\T M)$ sending $\p$ to $\pi_*\C_\p$ will be called \emph{Taylor map}. Its restriction to the set of jets which hold $\widehat{\p'}\subseteq\p$ will be named \emph{Taylor injection}.
 \end{defi}

\begin{obss}
Classical jets, $\p_m^\ell$, hold the condition $\widehat{\p'}\subseteq\p$. In the classical case, the Taylor injection assigns to each system of partial differential equations of order $\ell$ with $m$ variables and $n-m$ unknowns (= submanifold of $\M_m^\ell$, when we will dispose of differentiable structures) a system of partial differential equations of first order with $m$ independent variables and $(n-m){{m+\ell-1}\choose{m}}$ unknowns  (system of equations in $\M_m^{\ell-1}$, of first order, with  $m$ independent variables: submanifold of $\textrm{Grass}_{\, m}(T\M_m^{\ell-1})$).
\end{obss}
\bigskip

\section{$A$-points manifold}\label{Apuntos}

\begin{Terminologia y notaciones}\label{Apuntosnotaciones}
let $A$ be a Weil algebra, $M$ a manifold. An \emph{$A$-point of $M$} is a morphism of $\R$-algebras $p^A\colon\A\to~A$; the image of $f$ by $p^A$ will be denoted by $p^A(f)$ or $f(p^A)$. The set of all of the $A$-points of $M$ will be denoted $M^A$; the subset of the \emph{regular or proper $A$-points} (exhaustive morphisms), $\check M^A$. For cases  $A=\R_m^\ell$ we will put $M_m^\ell$ instead of $M^{\R_m^\ell}$. The composition of morphism $p^A$ with the taken of quotient $A\to~A/\m_A=\R$ is a point $p\in M$; following Weil, we will say that $p^A$ is an \emph{$A$-point near to $p$}.

Each $f\in\A$ determines the function (same notation) $f\colon M^A\to A$, by: $f(p^A)=p^A(f)$. If $\{a^\alpha\}$ is a basis of $A$ as a $\R$-vector space, each $f$ determines real functions $f_\alpha\colon M^A\to\R$ by the rule $f(p^A)=f_\alpha(p^A)a^\alpha$. Such $f_\alpha$'s will be called \emph{real components of $f$} with respect to the basis $\{a^\alpha\}$.
\end{Terminologia y notaciones}
\begin{teo}\label{diferenciablesApuntos}
$M^A$ admits  an structure of differentiable manifold by requiring the following condition: for each $f\in \A$, the real components $f_\alpha\colon M^A\to\R$ are in the class $\C^\infty$. For that structure, $\check M^A$ is either the empty set or a dense open set of $M^A$.
\end{teo}
\begin{proof}
Let $U$ be an open set in $M$ coordinated by functions $x^1,\dots, x^n$, which we can assume in $\A$ by reducing $U$ if necessary. As a function $U^A\to A$ we have $x^i=x^i_\alpha a^\alpha$, once chosen the basis of $A$. For each $f\in\A$ and each $p\in U$, the Taylor expansion gives us:
$$f\equiv\sum_{|\lambda|\le \ell}\frac 1{\lambda!}\left(D^\lambda f\right)(p)\,(x-x(p))^\lambda\quad \textrm{mod}\,\m_p^{\ell+1},$$
where $\ell=$order of $A$, $\lambda=(\lambda_1,\dots,\lambda_n)$, $|\lambda|=\lambda_1+\cdots\lambda_n$, etc.

Taking values in $p^A$, near to $p$:
$$f(p^A)=\sum_{|\lambda|\le \ell}\frac 1{\lambda!}\left(D^\lambda f\right)(p)\,\left(\sum_\alpha x_\alpha(p^A)a^\alpha-x(p)\right)^\lambda=
  \sum_\beta f_\beta(p^A)a^\beta.$$
   Expanding the terms into brackets, applying the multiplication table of $A$ and, finally, equating coefficients of each $a^\beta$ in both members of equation, we get $f_\beta$ as a polynomial in the $x_\alpha^i$'s with coefficients in $\A$, which are expressible as linear combinations, with constant coefficients, in the derivatives of $f$ of orders $\le\ell$, these combinations being the same for all $f$.

   On the other hand, the morphism $\R[x^1,\dots,x^n]\to A$ can be defined by assigning arbitrary values in $A$ to each variable $x^i$. If these values are attributed in such a way that, when composed with the taking of quotient $A\to A/\m_A=\R$, $(x^1,\dots,x^n)$ goes to a point in $U$, we obtain a morphism $p^A\colon\C^\infty(U)\to A$ near to a point $p\in U$. Thus, can be seen that $U^A\simeq U\times (\m_A)^{\times n}$ as sets. This one-to-one correspondence allows us to transport till $U^A$ the differentiable structure of  $U\times (\m_A)^{\times n}$. The $x_\alpha^i$'s are coordinates and each $f\in\A$ is a differentiable map $U^A\to A$. The uniqueness  of the differentiable structure on $U^A$ is assured by that condition.

   Given two coordinates subsets $U$, $V$, the manifolds $U^A$, $V^A$, induce on $(U\cap V)^A$ the same differentiable structure, which gives for $M^A$ the existence and uniqueness of the differentiable structures as set out in the statement.

   With the previous notation, for $p^A\in U^A$ to be a regular $A$-point is necessary and sufficient that the values at $p^A$ of the functions $x^i-x^i(p)$ $(i=1,\dots,n)$ generate $\m_A/\m_A^2$. Therefore, the fact of $p^A$ being no regular is translated into a system of algebraic equations between the real components $x_\alpha^i$ of the $x^i$. From here, we derive that $\check M^A$ is either empty or dense.
\end{proof}

\begin{Functorialidad}

Each morphism of manifolds $\vphi\colon M\to N$ gives a map (same notation) $\vphi\colon M^A\to N^A$, by $\vphi(p^A)=p^A\circ\vphi^*$. Translated into real components, we see that $\vphi^*(f_\alpha)=(\vphi^* f)_\alpha$, from which it follows that $\vphi\colon M^A\to N^A$ is a differentiable map for the structures given in  \ref{Apuntosnotaciones}. It results that the assignation $M\rightsquigarrow M^A$ is a covariant functor from the category of manifolds to itself.

Given a morphism $\alpha\colon A\to B$ of Weil algebras, the canonical map $M^A\to M^B$ ($p^A\to p^B=\alpha\circ p^A$) is smooth, as it is seen  immediately by expressing  $\alpha$ in terms of basis of $A$, $B$ as $\R$-vector spaces. In this way, the assignation $A\rightsquigarrow M^A$ is a covariant functor from the category of Weil algebras to that of manifolds.
\end{Functorialidad}

\begin{Modulo tangente en un Apunto}

For each $p^A\in M^A$ and each field $D$ on $M$, we will call  \emph{value of $D$ at $p^A$} 
the map $D_{p^A}\colon\A\to A$, $D_{p^A}f:=(Df)(p^A)$. This map is a derivation from the ring of functions $\A$ with values in $A$, which is considered as an $\A$-module by means of the morphism $p^A$.

Each derivation from $\A$ to $A$ at the point $p^A$ comes in this fashion from a field $D$ on $M$; the proof is performed by means of an argument similar to the one we used in the proof of \ref{tangentejets}.

We will denote by $T_{p^A}M$ the $\A$-module of derivations from $\A$ to $A$, at the point $p^A$. We will call it \emph{tangent module to  $M$ at the point $p^A$}. It is an $A$-module locally free of rank $=\dim M$, and basis $(\partial/\partial x^1)_{p^A}$,$\dots$, $(\partial/\partial x^n)_{p^A}$, in local coordinates $x^1,\dots,x^n,$ for $M$.
\end{Modulo tangente en un Apunto}

\begin{teo}\label{tangenteApuntos}
Tangent module $T_{p^A}M$ and tangent space $T_{p^A}M^A$ to the manifold $M^A$ at the point $p^A$ are canonically isomorphic $\R$-vector spaces: to the derivation $D_{p^A}$ it corresponds the tangent vector $\overline D_{p^A}$ such that
$\overline D_{p^A}f_\alpha=(D_{p^A}f)_\alpha$, for each $f\in\A$ and each basis $\{a^\alpha\}$ chosen in $A$.
\end{teo}
\begin{proof}
Chosen a basis of $A$ as a $\R$-vector space, multiplication table in $A$ is of the form $a^\alpha a^\beta=c_\gamma^{\alpha\beta}a^\gamma$, with the  $c$'s in $\R$. For each pair $f$, $g$, it follows that $(f\cdot g)_\gamma=c_\gamma^{\alpha\beta}f_\alpha g_\beta$. If $\overline D_{p^A}\in T_{p^A}M^A$, the rules of derivation for sums and products give that the map $D_{p^A}\colon\A\to A$ defined by $D_{p^A}f=(\overline D_{p^A}f_\alpha)a^\alpha$ is a derivation at $p^A$. The assignation $\overline D_{p^A}\mapsto D_{p^A}$ from $T_{p^A}M^A$ to $T_{p^A}M$ is an injective  morphism of vector spaces with the same dimension, $n\cdot\dim A$, and then, an isomorphism.
\end{proof}

\begin{cor}\label{independientesApuntos}
Let $\alpha\colon A\to B$ be an exhaustive  morphism of Weil algebras. For each manifold $M$, the corresponding morphism of tangent spaces
$T_{p^A}\check M^A\to T_{p^B}\check M^B$ ($p^B=\alpha\circ p^A$) is exhaustive. As a consequence, if the functions $F_1,\dots,F_k\in\C^\infty(\check M^B)$ are functionally independent, also so are their images $\alpha^*(F_1),\dots,\alpha^*(F_k)$ in $\C^\infty(\check M^A)$.
\end{cor}
\begin{proof}
For each tangent field $D$ on $M$ it holds $D_{p^B}=\alpha_*D_{p^A}$, and we conclude from the above theorem.
\end{proof}

Theorem \ref{tangenteApuntos} is a particular case of the following one.
\begin{teo}[Weil \cite{Weil}]
For an arbitrary manifold $M$ and Weil algebras $A$, $B$, we have a canonical isomorphism of manifolds:
$$(M^A)^B\simeq M^{A\otimes B}.$$
\end{teo}
\begin{proof}
Each $f\in\A$ defines the function $f\colon M^A\to A$, which determines a ring morphism
$$\A\To\C^\infty(M^A)\otimes_\R A,\quad f\mapsto f_\alpha\, a^\alpha,$$
with the notation we are using previously. We have consecutive morphism of rings
$$\A\To\C^\infty(M^A)\otimes A\To \C^\infty((M^A)^B)\otimes B\otimes A$$
which give, by composition, the morphism
\begin{equation*}
\xymatrix @R=1pt{
\A\ar[r]  &     \C^\infty((M^A)^B)\otimes B\otimes A \ar[r]     &      \C^\infty((M^A)^B)\otimes A\otimes B\\
f\ar@{|->}[r]    &     {f_\alpha\,a^\alpha\to(f_\alpha)_\beta\, b^\beta\otimes a^\alpha} \ar@{|->}[r]      &    (f_\alpha)_\beta\otimes a^\alpha\otimes b^\beta
}
\end{equation*}

 For each point $(p^A)^B\in(M^A)^B$, ``to take values at $(p^A)^B$'' is a morphism from $\C^\infty((M^A)^B)$ to $\R$, which gives, by composition with the previous one, a morphism $\A\to A\otimes B$, which is a $A\otimes B$-point $p^{A\otimes B}$ on $M$. By the above computation, for each $f\in\A$ is $(f_\alpha)_\beta((p^A)^B)=f_{\alpha\beta}(p^{A\otimes B})$ (taking $\{a^\alpha\otimes b^\beta\}$ as a basis in $A\otimes B$), and so the proof is concluded.
\end{proof}

\begin{El teorema es un caso particular del teorema de Weil}
When $B=\R_1^1=\R[x]/(x^2)$, a $B$-point is a map
$f\mapsto f(p_1^1)=f(p)+\epsilon\, D_pf$, where $f(p)$ and $D_pf$ are real numbers and  $\epsilon=[x]_{\textrm{mod}\,(x^2)}$. This map being morphism of  $\R$-algebras means that $p$ is a point in $M$ and $D_p$ is a vector tangent to $M$ at $p$. Thus, $M_1^1=TM$ is the tangent bundle of $M$.

In the same way, it can be seen that, if $A$ is an arbitrary Weil algebra, an $A\otimes\R_1^1$-point of $M$ is a pair $(p^A,D_{p^A})$ where $p^A$ is an  $A$-point and $D_{p^A}$ is a derivation of $\C^\infty(M)$ at $p^A$ (valuated in $A$). On the other hand, a point of $(M^A)_1^1$ is a $p^A\in M^A$ and a derivation of $\C^\infty(M^A)$ (valuated in $\R$) at $p^A$.

In this way, Weil theorem gives \ref{tangenteApuntos} as the particular case when $B=\R_1^1$.
\end{El teorema es un caso particular del teorema de Weil}

\begin{Producto cartesiano de Apuntos Prolongacion de grupos de Lie y fibrados asociados}

Let  $M$ and $N$ be smooth manifolds and let $A$ be a Weil algebra. Given $A$-points $p^A$ in $M$, $q^A$ in $N$, we have a morphism
\begin{equation*}
\xymatrix {\displaystyle{\A\otimes_\R\B}\ar[rr]^-{p^A\otimes q^A} && A\otimes A}
\end{equation*}
which can be composed with the morphism ``multiplication'' $A\otimes A\to A$ so giving a morphism
$\A\otimes_\R\B\to A$.

This morphism can be prolonged by continuity (or, simply, according to its Taylor expansion) to a morphism:
$$\C^\infty(M\times N)\To A$$
which we will denote by $p^A\times q^A\in (M\times N)^A$, and will be called \emph{cartesian product of the $A$-points} $p^A$, $q^A$. That $A$-point is the only one on $M\times N$ which, when restricted to the subrings $\A$, $\B$ of $\C^\infty(M\times N)$ coincide with $p^A$, $q^A$, respectively.

When $A=\R$, $p\times q$ is the couple $(p,q)\in M\times N$.

When an ``action'' of $M$ on $N$ is given, that is to say, a morphism $M\times N\to N$ (for instance, when $N$ is a principal fiber bundle with structural group $M$), such an action can be lifted to the spaces of  $A$-points by means of the cartesian product:
$$M^A\times N^A\To N^A,\quad (p^A,q^A)\mapsto p^A\times q^A\mapsto p^A\cdot q^A,$$
where $p^A\cdot q^A$ is the image of $p^A\times q^A\in (M\times N)^A$ under the morphism $M\times~N\to~N$.

In particular, if $G$ is a Lie group, $G^A$ is also a Lie group. The identity map of $G^A$ is the morphism ``to take values at the unity $e\in G$'', $f\to f(e)$.  Indeed, for each $A$-point $g^A\in G^A$, the $A$-point in $G\times G$ which coincides with $g^A$ in the first factor (in $\C^\infty(G)\otimes 1$) and coincides with ``to take values at $e$'' in the second one, say $e^A$, has as product $g^A\cdot e^A$ the $A$-point of $G$ which assigns to each $f\in\C^\infty(G)$ the value of $f(\overline g\cdot e)=f(\overline g)$ at $g^A$; so, $g^A\cdot e^A=g^A$.

The inverse element of  $g^A$ is its image by the morphism of manifolds ``inversion'': $G\to G$.
\end{Producto cartesiano de Apuntos Prolongacion de grupos de Lie y fibrados asociados}

\begin{Un ejemplo el grupo tangente a un grupo de Lie}
Let $G$ be a Lie group as above. Its tangent bundle $TG=G_1^1$ is also a Lie group, whose group law  will be detailed  now.

Let us denote by $\Phi\colon G\times G\To G$ the morphism of manifolds given by the multiplication rule. In local coordinates, the image of $(x,y)=(x^1,\dots,x^n,y^1,\dots,y^n)$ is $z=\Phi(x,y)$. Let $p_1^1=(p,D_p)$, $q_1^1=(q,\overline D_q)$ be points of $G_1^1$. The morphism $p_1^1\cdot q_1^1$ sends the coordinate $z^i$ to
\begin{align*}
\Phi^i(x(p_1^1),y(q_1^1))&=\Phi^i(x(p)+\epsilon D_px,\, y(q)+\epsilon\overline D_qy)\\
                         &= \Phi^i(x(p),y(q))\\
                         &\phantom{=\Phi^i(x}+\epsilon
                           \left[\frac{\partial\Phi^i}{\partial x^j}(x(p),y(q))D_px^j+
                           \frac{\partial\Phi^i}{\partial y^k}(x(p),y(q))\overline D_qy^k\right]\\
                         &=z^i(p,q)+\epsilon \left[R_q D_p+L_p\overline D_q\right] z^i
\end{align*}
where $R_q$ is the right translation by $q$ and $L_p$ is the left translation by $p$ on $G$.

It is derived that the multiplication operation on $G_1^1$ is
$$(p,D_p)\cdot (q,\overline D_q)=(p\cdot q,R_qD_p+L_p\overline D_q).$$

The identity element for the multiplication is $(e,0)$. The inverse element of $(p,D_p)$ is:
$(p,D_p)^{-1}=(p^{-1},-L_{p^{-1}}R_{p^{-1}}D_p)$.

 Let us use the fact of $G$ being parallelizable by means of left translations: each $D_p\in T_pG$ is the value at $p$ of a (unique) vector field $D$ of the Lie algebra of $G$: $D_p=L_p\delta$, with $\delta=D_e\in T_eG$. The inverse element of $(p,D_p)$ is, then,
 \begin{equation*}
 (p,D_p)^{-1} =(p^{-1}, -L_{p^{-1}}R_{p^{-1}}L_p\delta)=
     (p^{-1},-L_{p^{-1}}(\textrm{Ad}\,p)(\delta))
  \end{equation*}

  By identifying each vector tangent to $G$ at a point with its associated vector field in the Lie algebra, we have
  \begin{align*}
  (p,\delta)^{-1} &=(p^{-1}, -\left(\textrm{Ad}\,p)\,\delta\right)\\
  (p,\delta)\cdot(q,\overline\delta)&=\left(p\cdot q,(\textrm{Ad}\,q^{-1})\,\delta+\overline\delta\right).
  \end{align*}
\end{Un ejemplo el grupo tangente a un grupo de Lie}
\bigskip

\section{Differentiable structures on the $A$-jet spaces}\label{estructuradiferenciable}

Let $A$ be a Weil algebra of order $\ell$ and width   $m$. For arbitrary $n\ge m$ there exist exhaustive morphisms $\alpha\colon\R_n^\ell\to A$; for instance, if $a^1,\dots,a^m$  are elements of $\m_A$ whose classes $\textrm{mod}\,\m_A^2$ are linearly independent, and $x^1,\dots,x^n$ are elements chosen in $\R_n^\ell$, as above, we can define $\alpha$ by the rule $\alpha(x^i)=a^i$ ($i=1,\dots,m$), $\alpha(x^{m+j})=0$ ($j=1,\dots,n-m$).

\begin{lemaN}\label{lemaepimorfismo}
Once chosen $\alpha\colon \R_n^\ell\to A$, any other epimorphism $\beta\colon \R_n^\ell\to A$ is of the form $\beta=\alpha\circ g$, where $g$ is certain automorphism of $\R_n^\ell$.
\end{lemaN}
\begin{proof}
Let us take for $\alpha$ the notation as above. Let $y^1,\dots,y^m\in\R_n^\ell$ be such that $\beta(y^i)=a^i$. The classes of $y^1,\dots,y^m\,\,\textrm{mod}\,\m_{\R_n^\ell}^2$ are linearly independent because also so are their images in $\m_A/\m_A^2$. We can add elements $y^{m+1},\dots,y^n$ in $\m_{\R_n^\ell}$ which, joint to $y^1,\dots,y^m$ give a generating system of $\R_n^\ell$; by subtracting from each $y^{m+j}$ a polynomial in $y^1,\dots,y^m$ we can get that $\beta(y^{m+j})=0$ ($j=1,\dots,n-m$). Next, we define $g\colon{\R_n^\ell}\to{\R_n^\ell}$ by the rule $g(y^i)=x^i$. We have $\beta=\alpha\circ g$.
\end{proof}

\begin{La estructura diferenciable de MA}

Let $M$ be a smooth manifold of dimension $n\ge m$. The choice of an epimorphism $\alpha\colon\R_n^\ell\to A$ determines a map
$$\pi_\alpha\colon\check M_n^\ell\To\check M^A,\quad \pi_\alpha(p_n^\ell)=\alpha\circ p_n^\ell,$$
which is exhaustive. In fact, a regular $A$-point factorizes into the composition of passing to the quotient  $\A\to\A/\m_p^{\ell+1}$ and the exhaustive morphism
${ p}^A\colon\A/\m_p^{\ell+1}\to A$. Once chosen any $p_n^\ell\in\check M_n^\ell$, the composition ${ p}^A\circ\left({ p}_n^\ell\right)^{-1}\colon\R_n^\ell\to A$ is, by Lema \ref{lemaepimorfismo}, of the form $\alpha\circ g$, with $g\in\textrm{Aut}\,\R_n^\ell$. It follows that  $p^A=\alpha\circ g\circ p_n^\ell=\pi_\alpha(g\circ p_n^\ell)$.

 Two points $p_n^\ell, \overline p_n^\ell\in\check M_n^\ell$ over the same  $p\in M$ are related in the form $\overline p_n^\ell=g\circ p_n^\ell$, with  $g\in\textrm{Aut}\,\R_n^\ell$ univocally determined by the pair $(p_n^\ell,\overline p_n^\ell)$. We have $\pi_\alpha(\overline p_n^\ell)=\alpha\circ g\circ p_n^\ell$, and, being $p_n^\ell$ exhaustive, the necessary and sufficient for that $\pi_\alpha(\overline p_n^\ell)=\pi_\alpha(p_n^\ell)$  holds, is $\alpha\circ g=\alpha$.

 Let us denote by $G$ the group $\textrm{Aut}\, \R_n^\ell$, $G_\alpha$ the subgroup comprised by those $g\in G$ such that $\alpha\circ g=\alpha$. Then, we have show that the projection $\pi_\alpha\colon \check M_n^\ell\to\check M^A$ is exhaustive and its fibres are the orbits of the subgroup $G_\alpha$ of $G$ when acts on $\check M_n^\ell$.

  When $\check M_n^\ell$ and $\check M^A$ are endowed with the structure of smooth manifold by Theorem \ref{diferenciablesApuntos}, the projection $\pi_\alpha$ is a differentiable map and has maximal rank (corollary \ref{independientesApuntos}).

  Let $U$ be an open subset of $M$ coordinated by functions $x^1,\dots,x^n$. For each $p\in U$, the classes $x^i-x^i(p)\,\textrm{mod}\,\m_p^{\ell+1}$ are generators of $\A/\m_p^{\ell+1}$. The choice of these generators establishes an isomorphism between the algebras $\A/\m_p^{\ell+1}$ and $\R_n^\ell$. When $p$ runs over $U$, the choice of coordinates gives us a section of the fiber bundle $\check U_n^\ell\to U$, precisely the section which sends $p\in U$ to $\xi(p)\colon\C^\infty(U)\to\R_n^\ell$, where $\xi(p)$ is the composition of the pass to the quotient with the isomorphism $\C^\infty(U)/\m_p^{\ell+1}\simeq\R_n^\ell$ established by the choice of coordinates, as said above.

Each point $p_n^\ell\in\check U_n^\ell$ is of the form $p_n^\ell=g\circ\xi(p)$, for certain $g\in G$. Thus, it is established a biunivocal correspondence
$$\check U_n^\ell\simeq U\times G,$$
which is a differentiable isomorphism (as coordinates on  $G$ serve the functions which in each $g\in G$ take the value $x_\alpha^i(g\circ\xi(p))$, $i=1,\dots,n$, $0<|\alpha|\le\ell$).

The projection $\pi_\alpha\colon\check M_n^\ell\to\check M^A$ puts in biunivocal correspondence $\cM^A$ with the set of orbits of $G_\alpha$ when it acts on $\cM_n^\ell$. In the open set $U$:
$$\cU^A\simeq U\times G/G_\alpha,\quad p^A\leftrightarrow (p,[g]_{\textrm{mod}\,G_\alpha})=(p,G_\alpha\cdot g)$$
by means of the choice of coordinates in $U$ and the Weil algebras morphism $\alpha$: $p^A=\alpha\circ g\circ\xi(p)$.

We must to check that both, the smooth structure  which is own of $\cU^A$, given in Theorem \ref{diferenciablesApuntos}, and the one given by the isomorphism of sets $\cU^A\simeq U\times G/G_\alpha$ coincide. Indeed, in the projection $\cU_n^\ell\to \cU^A$, the smooth functions on $\cU^A$ lift to $\cU_n^\ell$ as smooth functions which are invariants under the action of the group $G_\alpha$; for that, they are smooth functions on the space of orbits $U\times G/G_\alpha$. In addition, a family of such functions in $\cU^A$ being functionally independent will remain so when lifted to $\cU_n^\ell$ (corollary \ref{independientesApuntos}) and we get the statement.

Let us consider the $A$-point $p^A=\alpha\circ p_n^\ell\in \cM^A$; its kernel is the jet $\p^A=\textrm{Ker}\,p^A=(p_n^\ell)^{-1}(\textrm{Ker}\,\alpha)$. Any other regular $A$-point in the same $p\in M$ is of the form $\overline p^A=\alpha\circ g\circ p_n^\ell$, with $g\in G$. The jet of $\overline p^A$ is $\overline\p^A=(p_n^\ell)^{-1}g^{-1}(\textrm{Ker}\,\alpha)$. The necessary and sufficient condition for  $\overline\p^A=\p^A$ is that it holds $g^{-1}(\textrm{Ker}\,\alpha)=\textrm{Ker}\,\alpha$, that is to say, that the automorphism $g$ of $\R_n^\ell$ will transform the ideal $\textrm{Ker}\,\alpha$ into itself. Let us denote $I=\textrm{Ker}\,\alpha$ and let $G_I\subseteq G$ be the subgroup comprised by the  $g\in G$ such that $g(I)=I$. We have the inclusions
$$G_\alpha\subseteq G_I\subseteq G=\textrm{Aut}\,\R_n^\ell.$$

Each $g\in G_I$ transforms pairs of congruent elements  $\textrm{mod}\,I$ into pairs of congruent elements $\textrm{mod}\,I$, and therefore, it projects as an automorphism of $A=\R_n^\ell/I$. The condition for  $g$ to be projected to the  identity is  $\alpha\circ g=\alpha$, that is to say, $g\in G_\alpha$. Thus, we have a morphism of groups $G_I\to\textrm{Aut}\, A$ whose kernel is $G_\alpha$. We derive that $G_\alpha$ is a normal subgroup of $G_I$, and we get an injection
$$G_I/G_\alpha\hookrightarrow \textrm{Aut}\, A.$$

Let us see that the previous injection is an isomorphism: let $\widetilde g\in\textrm{Aut}\, A$ and put $\beta:=\widetilde g\circ\alpha$. By Lemma \ref{lemaepimorfismo}, there exists a $g\in G$ such that $\widetilde g\circ\alpha=\alpha\circ g$:
\begin{equation*}
\xymatrix{
\R_n^\ell\ar[r]^-g \ar[d]_-\alpha & \R_n^\ell\ar[d]^-\alpha\\
A\ar[r]_-{\widetilde g} & A
}
\end{equation*}
Obviously, we have  $g\in G_I$, and the image of $g$ in $\textrm{Aut}\, A$ is $\widetilde g$.

Inasmuch  that $A$-points $p^A=\alpha\circ p_n^\ell$, $\overline p^A=\alpha\circ\overline p_n^\ell$ have the same jet if and only if $\overline p_n^\ell=g\circ p_n^\ell$ for a suitable $g\in G_I$, the set of jets of type $A$ in $M$, $\M^A$, is the set of orbits of the group $G_I$ acting on $\cM_n^\ell$. It is also the set of orbits of $\textrm{Aut}\, A$ acting on $\cM^A$.

Let us return to consider the open set $U$ of $M$, coordinated by $x^1,\dots,x^n$. The choice of coordinates and the morphism $\alpha$ give sections $\xi$ of $\cU_n^\ell\to U$ and $\alpha\circ\xi$ of $\cU^A\to U$. By taking the kernel of the last one, we have  a section $\textrm{Ker}\,\circ\alpha\circ\xi$ of $J^A(U)\to U$. By means of the local sections, the maps $\cM_n^\ell\to\cM^A\to \M^A$, restrict themselves,  on the open set $U$, to:
$$\cU_n^\ell\simeq U\times G\To\cU^A\simeq U\times G/G_\alpha\To J^A(U)\simeq U\times G/G_I.$$

We can transport to $J^A(U)$ the differentiable structure of $U\times G/G_I$; with that, the smooth functions on $J^A(U)$ are the smooth functions on $\cU_n^\ell$ that are invariants under the action of $G_I$, which is a closed subgroup of $G$. Such functions, because they are invariants under the  subgroup $G_\alpha$ of $G_I$, are functions on $\cU^A$, invariants under the action of the group $G_I/G_\alpha=\textrm{Aut}\, A$.

By taking a covering of $M$ by coordinated open sets, it is obtained the differentiable structure on $\M^A$.
\end{La estructura diferenciable de MA}

It is classical that (for arbitrary Lie group $G$ and its closed subgroup $G_I$) the morphism of manifolds $G\to G/G_I$ admits local sections which locally trivialize the fibration. Each one of these sections, composed with $\alpha$, give a local section for $\cU^A\to J^AU$, that makes of $\cU^A$ a principal fiber bundle with structural group $\textrm{Aut}\,A=G_I/G_\alpha$ and base manifold $J^AU$. Being it so for each coordinate open  $U$ of $M$, $\cM^A\to \M^A$ is a principal fiber bundle with structural group $\textrm{Aut}\,A$. We can state:
\begin{teo}\label{jetfibradoprincipal}
For each Weil algebra $A$ and each smooth manifold $M$ of dimension $\ge$ width of $A$, $\cM^A$ and $\M^A$ have smooth manifold structure. That of $\cM^A$ is that described in Theorem \ref{diferenciablesApuntos}. That of $\M^A$ is determined by the structure of  $\cM^A$ and the canonical projection $\textrm{Ker}\colon\cM^A\to \M^A$, which makes of $\cM^A$ a principal fiber bundle with base manifold $\M^A$ and structural group $\textrm{Aut}\,A$.

Once chosen an exhaustive morphism $\alpha\colon\R_n^\ell\to A$, on each coordinate open set $U$ of $M$, we have sections $\xi$ of $\cU_n^\ell\to U$, $\alpha\circ\xi$ of $\cU^A\to U$, $\textrm{Ker}\circ\alpha\circ\xi$ of $J^A(U)\to U$ which establish isomorphisms of fiber bundles on $U$:
$$\cU_n^\ell\simeq U\times G,\quad \cU^A\simeq U\times G/G_\alpha,\quad J^A(U)\simeq U\times G/G_I,$$
where $G=\textrm{Aut}\,A$, $G_\alpha=\{g\in G\,|\,\alpha\circ g=\alpha\}$, $G_I=\{g\in G\,|\,g(I)=I\}$, $I=\textrm{Ker}\,\alpha$ and $G_I/G_\alpha=\textrm{Aut}\,A$.
\end{teo}

 The projection of manifolds $\cM^A\to \M^A$ gives epimorphisms of tangent spaces $T_{p^A}\cM^A\to T_{\p^A}\M^A$, when $\p^A=\textrm{Ker}(p^A)$; vectors in the kernel of that epimorphism are those tangent to the fiber of $\p^A$, which are the derivations $\in T_{p^A}\cM^A$ that annihilate all the functions $\in\C^\infty(\cM^A)$ vanishing on the fiber; these functions are those of the ideal generated by the real components of the functions $f\in\p^A$. From that and Theorem \ref{tangenteApuntos} it is derived the:
 \begin{teo}\label{tangenteAjets}
 The space $T_{\p^A}\M^A$ and the module $T_{\p^A}M$ are canonically isomorphic as $\R$-vector spaces.
 \end{teo}

 \begin{coment}
 For each $\p\in\M^A$, it holds  $\A/\p\simeq A$, and the value at $\p$  of the tangent space to $\p$, $\D(\p)$, is identified with the space  $\textrm{Der}\,(A,A)$, which is the Lie algebra of the group $\textrm{Aut}\,A$. Because  $\M^A$ is the space of orbits of $\textrm{Aut}\,A$ when acts on $\cM^A$, $T_{\p^A}\M^A$ is the quotient of $T_{p^A}\cM^A$ by $\textrm{Der}(A,A)$, and we obtain what was stated by \ref{tangenteAjets}.

 For analogous reasons, a Weil algebra morphism $\alpha\colon A\to B$ which determines a morphism $\check\alpha\colon\cM^A\to\cM^B$ (necessarily, $\alpha$ is exhaustive) passes to $\M^A\to \M^B$ when the ideal $I=\textrm{Ker}\,\alpha$ of $A$ is invariant under the action of the group $\textrm{Aut}\,A$, hence, any derivation $\delta\colon A\to A$ holds $\delta(I)\subseteq I$. For example, that occurs when $A=\R_m^\ell$ and  $B=\R_m^{\ell-1}$. More in general, when $\p$ is any jet in $M$, $A=\A/\p$ and $B=\A/\p'$ (Definition \ref{modulotangente} and Proposition \ref{derivadappprima}).
 \end{coment}

\begin{Prolongacion de ideales}

Let $M$ be a manifold and $A$ a Weil algebra. Fixed a basis $\{a^\alpha\}$ of $A$ as $\R$-vector space, for each $f\in\A$ we define the real components of $f\colon M^A\to A$ by $f(p^A)=f_\alpha(p^A)a^\alpha$. Given an automorphism  $g$ of $A$, expressed in coordinates by $g(a^\alpha)=c_\beta^\alpha a^\beta$, the action of $g$ on $M^A$ gives $f(g(p^A))=f_\alpha(p^A)c_\beta^\alpha a^\beta$, so that $g^*(f_\beta)=c_\beta^\alpha f_\alpha$. It is derived that the ideal of $\C^\infty(M^A)$ generated by the real components of $f$ is invariant under the group $\textrm{Aut}\,A$.

Let us consider the principal fiber bundle $\cM^A\to\M^A$ (\ref{jetfibradoprincipal}). By taking a local section, the specialization to it of the ideal generated by the real components of $f$ is transported to an ideal of the ring of functions of the open set of $\M^A$ where such a section  is defined; this ideal is independent of the chosen section. It follows that the ideals defined by two local sections determine the same ideal on the intersection of the respective open sets where these sections are defined. Therefore, there exists an ideal of $\C^\infty(\M^A)$ that generates in $\C^\infty(\cM^A)$ the same ideal as the ideal of real components of $f$.

Given an arbitrary ideal $I$ of $\A$, the previous process, when applied to each $f\in I$, determines in $\C^\infty(\M^A)$ an ideal, which be called \emph{prolongation of $I$ to $\C^\infty(\M^A)$}. This ideal is characterized by the fact of generating in $\C^\infty(\cM^A)$ the same  ideal as the collection of all the real components of all the $f\in I$.

 Functions does not prolong from $M$ to $\M^A$ in the sense in what they do it to $\cM^A$. But, where appropriate, we can use local sections of $\cM^A\to\M^A$ in order to talking about ``real components'' of $f$ on $\M^A$, and we will handle them in local computations.

\end{Prolongacion de ideales}

\begin{Prolongacion de jets Relacion con los sistemas de contacto}

As any ideal, a jet $\p$ can be prolonged as an ideal to each manifold $\M^A$. For example, if $\p$ is of type $A$, its prolongation to $\C^\infty(\M^A)$ is $\m_\p$, the maximal ideal of the point $\p\in\M^A$. That is a tautology, by taking into account that the fiber of $\cM^A\to \M^A$ on $\p$ consists of the points $p^A\in\cM^A$ where vanish all the $f\in \p$, and no one  $f\in\A$ out of $\p$ has this property.
\end{Prolongacion de jets Relacion con los sistemas de contacto}

Now let us consider a jet $\p'$, of type $A'$, such that  $\p\subseteq\p'$ and a projection $\alpha\colon A\to A'$ such that $\textrm{Ker}\,\alpha$ remains stable under $\textrm{Der}(A,A)$. Then, there exists a map of tangent modules $T_\p M\to T_{\p'}M$ and the contact system $\Omega_{(\p,\p')}$ consists of the submodule $\{d_{\p'}f\,|\,f\in\p\}=d_{\p'}\p$:
\begin{teo}\label{prolongacionycontacto}
Let $\p\subseteq\p'$ be jets of types $A$, $A'$ such that $\textrm{Der}(A,A)$ leaves stable the kernel of the projection $A\to A'$. Then, the contact system $\Omega_{(\p,\p')}$ considered as a subspace of $T^*_{\p'}\M^{A'}$, coincides with the space of differentials of the prolongation of $\p$ to $\p'$. In particular, if $\p'$ is the derived ideal of $\p$, the real components of $\Omega_\p$ are the differentials at $\p'$ of the real components of the prolongation of $\p$ to $\C^\infty(\M^{A'})$ (we have put $A=\A/\p$, $A'=\A/\p'$).
\end{teo}
\bigskip

\section{Differentiable structure of the contact system}
\begin{teo}
\leavevmode
\begin{enumerate}
\item Let $A$, $A'$ be Weil algebras, $\alpha\colon A\to A'$ an epimorphism such that the group $\textrm{Aut}\,A$ leaves stable the ideal $\textrm{Ker}\,\alpha$. Let $\pi_\alpha\colon\M^A\to\M^{A'}$ be the morphism of manifolds associated with $\alpha$. For each $\p\in\M^A$, let $\p'=\pi_\alpha(\p)$ and  $\Omega_{(\p,\p')}\subseteq T^*_\p M$. The real components of  $\Omega_{(\p,\p')}$, lifted to $T^*_\p\M^A$  constitute, when $\p$ runs over $\M^A$, a Pfaff system of class $\C^\infty$ on $\M^A$. In particular, the contact system $\Omega$ on $\M$ has class $\C^\infty$ on each manifold $\M^A$. The rank of $\Omega$ on $\M^A$ is the codimension of $\X^{A'}$ en $\M^{A'}$, where $X$ is any submanifold of dimension $=$ width of $A$ on $M$.
\item If the jets of type  $A$ hold the condition $\widehat\p'\subseteq\p$, the manifolds $\X^A$ are solutions of maximal dimension of the contact system $\Omega$ in $\M^A$.
\item In the conditions of item 2), the Pfaff system $\Omega$ es irreducible on $\M^A$.
\item In the conditions of item 2), in the Taylor  $\M^A\to J_\mu^1(\M^{A'})$ ($\mu=\dim\X^{A'}$, for $\dim X=m$), the contact system on $J_\mu^1(\M^{A'})$ specializes to $\M^A$ as the contact system on $\M^A$.
\end{enumerate}
\end{teo}
\begin{proof}
\leavevmode
\begin{enumerate}

\item
By taking into account Theorem \ref{prolongacionycontacto}, the problem reduces itself to prove that, locally on $\M^A$, there is a finite family of polynomials $P_\nu(x^1,\dots,x^n)$ in the local coordinates $x^i$ of $M$, with coefficients in $\C^\infty(\M^A)$ such that, when their numerical values are fixed at each $\p$, it is obtained a collection of generators of $\p$ as an ideal of $\A$.

Let us consider the generic polynomial  $P(\lambda,x)$ of degree $=$ the order of $A$, where $\lambda$ is the set of undetermined coefficients of the monomials in the $x^1,\dots,x^n$. Given a point $p^A\in\cM^A$, the condition for that the assignation of numerical values $\lambda(p^A)$ to the coefficients of $P$ to give $P(\lambda(p^A),x)$ at $\p=\textrm{Ker}(p^A)$ is that the coefficients $\lambda(p^A)$ to hold a system of linear equations (one for each $a^\alpha$ of a given basis of $A$ as $\R$-vector space) whose coefficients are given polynomials in the $x_\alpha^i(p^A)$ with constant coefficients. The rank of this system on $\cM^A$ is constant, equal to  $\dim A$. For that reason, the solutions $\lambda(p^A)$ are (rational ) functions of class $\C^\infty$ in the $x_\alpha^i$, linearly depending of a finite number of parameters which, when fixed, give each element of a basis of  $\p$ as an $\R$-vector space. Such functions are, obviously, invariant under  $\textrm{Aut}\,A$, so that they are functions on $\M^A$, and we conclude the argument.

\item
If we go back to the definitions we see that a  jet $\p\in\X$ is a jet of $\M$ such that $X$ passes trough $\p$ (\ref{terminologia}); a tangent vector $\DD_\p\in T_\p X$ is a vector  $\DD_\p\in T_\p M$ which is tangent to $X$ at $\p$ (Section \ref{terminologia2}); and $T_\p X$ is contained in the Cartan system $\C_\p$ for any jet of width $=\dim X$ (Definition  \ref{defcontacto}). For that, $\X^A$ is a solution of the contact system $\Omega$ de $\M$, for arbitrary Weil algebras $A$ of width $m=\dim X$ (Definition \ref{defcontacto}).

 Let $\Z\subset\M^A$ be a solution of $\Omega$ that contains  $\X^A$. For each $\p\in\X^A$ it holds $T_\p\Z\subseteq\C_\p$; thus, in the projection
 $\pi_*\colon T_\p M\to T_{\p'}M$, we have
 $\pi_* T_\p Z\subseteq\pi_*\C_\p=T_{\p'}\X^{A'}$ (we use the identification of $T_\p M$ with $T_\p\M^A$ given in \ref{tangenteAjets} and we have put $A'=\A/\p'$). As the converse inclusion holds, because $\X^A\subseteq\Z$, we have $\pi_*T_\p\Z=T_{\p'}\X^{A'}$. Therefore, in a neighborhood of $\p$ in $\Z$, the rank of $\pi_*$ at each $T_\q\Z$ is $\ge\dim\X^{A'}$, so that $=\dim\X^{A'}$ by Proposition \ref{proyeccioncontacto1}. Then,  $\pi_*T_\q\Z$ has constant dimension in a neighborhood of $\p$. By the rank theorem, $\pi\Z$ (in a neighborhood of $\p'$) is a manifold of dimension equal to that of  $\X^{A'}$, thereupon $\pi\Z=\X^{A'}$ in that neighborhood.

 Let us assume that (in the considered open set) there exists a  $\overline\q\in\Z$ out of $\X^A$. Its image in $\X^{A'}$ would be a point $\q'$ image of some $\q\in\X^A$. Then, the manifold $X$ passes trough $\q'$ but not through $\overline\q$; thus, by Theorem \ref{teoTaylor} and Proposition \ref{proyeccioncontacto1}, it should be  $T_{\q'}X\ne\pi_*\C_{\overline q}=\pi_*\Z_{\overline\q}$, against that proved before. That contradiction finishes the proof.

\item It is classical (Cartan) that the manifolds which are solutions of maximal dimension of a Pfaff system are generated by manifold-solutions of the characteristic system. In order to see that $\Omega$ is irreducible on $\M^A$, it is enough, then, to see that the solutions of maximal dimension of $\Omega$ passing through a given $\p$ have the zero vector as the only common tangent vector. Then, by item 2), it is sufficient to show that, given a manifold $X$ of dimension $m$ passing through $\p$, for each tangent field  $D$ on $M$ being tangent to $X$ at $\p$, with  $\DD_\p\ne 0$, there exists a manifold $Y$ of dimension $m$ which passes through  $\p$ but such that $\DD_\p$ is not tangent to $Y$. For that, we take local coordinates $x^1,\dots,x^m,y^1,\dots,y^r$ on $M$ adapted to $X$, in such a way that $I(X)=(y^1,\dots,y^r)$, and $\p$ is of the form $(y^1,\dots,y^r)+\m_p^{\ell+1}+(Q_h(x))$. By eliminating terms vanishing on $T_\p M$, a field $D$ tangent to $X$ at $\p$ can be written locally as a linear combination of the $\partial/\partial x^\alpha$ with coefficients which are polynomial in the $x$'s. If $\DD_\p\ne 0$, there exists $f\in\p$ such that $Df\notin\p$. The manifold defined by equations  $y^1=f,y^2=0,\dots,y^r=0$ passes through  $\p$ and  so $\DD$ is not tangent to $Y$ at $\p$.

\item From \ref{contactoaniquilado} and \ref{proyeccioncontacto}, $\Omega$ on $\M^A$ is, at each $\p\in\M^A$, the lifting, by ``pull-back'' of the forms on $\M^{A'}$ which annihilate  $T_{\p'}X^{A'}$, for any  $m$-dimensional submanifold $X$ of $M$ passing through $\p$. The same occurs in the image of $\p$ in $J_\mu^1(\M^{A'})$ with the contact system of that manifold (\ref{orden1}), and thus we conclude.
\end{enumerate}

\end{proof}



\end{document}